\documentclass{amsart}

\usepackage{float}
\usepackage{subfig, epsfig}
\usepackage{graphicx}

\newtheorem{theorem}{Theorem}[section]
\newtheorem{lemma}[theorem]{Lemma}
\newtheorem{proposition}[theorem]{Proposition}
\newtheorem{corollary}[theorem]{Corollary}
\theoremstyle{definition}
\newtheorem{definition}[theorem]{Definition}
\newtheorem{example}[theorem]{Example}

\theoremstyle{remark}
\newtheorem{remark}[theorem]{Remark}

\numberwithin{equation}{section}

\begin{document}

\title{The log-Brunn-Minkowski inequality in $\mathbb{R}^3$}

\author{Yunlong Yang}
\address{School of Science, Dalian Maritime University,
Dalian, 116026,  People's Republic of China}
%\curraddr{Department of Mathematics and Statistics,
%Case Western Reserve University, Cleveland, Ohio 43403}
\email{ylyang@dlmu.edu.cn}

\author{Deyan Zhang}
\address{School of Mathematical Sciences,
Huaibei Normal University, Huaibei, 235000, People's Republic of China}
\email{zhangdy8005@126.com}
\thanks{The first author was supported in part by the Doctoral Scientific
Research Foundation of Liaoning Province (No.20170520382) and the Fundamental Research
Funds for the Central Universities (No.3132017046).}
\thanks{The second author is the corresponding author
and was supported in part by the National Natural Science Foundation of China
(No.11671298, 11561020) and is partly supported by the Key Project of
Natural Science Research in Anhui Province (No.KJ2016A635).}
%    General info

%\date{}



\begin{abstract}
B\"{o}r\"{o}czky, Lutwak, Yang and Zhang recently proved the
log-Brunn-Minkowski inequality which is stronger than the
classical Brunn-Minkowski inequality for two origin-symmetric
convex bodies in the plane. This paper establishes the
log-Brunn-Minkowski, log-Minkowski, $L_p$-Minkowski
and $L_p$-Brunn-Minkowski inequalities
for two convex bodies in $\mathbb{R}^3$.
\end{abstract}

\maketitle

\section{Introduction}\label{sec1}

Let $\mathcal{K}^n$ be the set of all convex bodies, i.e., compact convex
sets with non-empty interior, in the $n$ dimensional Euclidean space
$\mathbb{R}^n$, and let $\mathcal{K}_0^n$ be the class of
members of $\mathcal{K}^n$ containing the origin in their interiors.
For two convex bodies $K,\,L\in \mathcal{K}^n$,
the classical Brunn-Minkowski inequality states that
\begin{equation*}
    V((1-\lambda)K+\lambda L)^{\frac{1}{n}}\ge (1-\lambda)V(K)^{\frac{1}{n}}
    +\lambda V(L)^{\frac{1}{n}},
\end{equation*}
with equality if and only if $K$ and $L$ are homothetic,
where $(1-\lambda)K+\lambda L=\{(1-\lambda)x+ \lambda y\mid x\in K,\, y\in L\}$
is the Minkowski combination of $K$ and $L$, and $V(\cdot)$ denotes the
$n$-dimensional volume (i.e. Lebesgue measure) functional.
The Brunn-Minkowski inequality is an extremely powerful
tool and plays a significant role in convex geometry,
its various aspects can be found in Gardner's article
\cite{G2002} and Schneider's monograph \cite{S2014}.

In the early 1960s, Firey \cite{F1962} (see also \cite{S2014})
generalized the Minkowski combination of convex bodies to the
$L_p$-Minkowski combination for each $p\geq 1$. In the 1999s,
Lutwak \cite{L1993,L1996} showed that many classical results
can be extended to the $L_p$ Brunn-Minkowski-Firey theory.
Recently, B\"{o}r\"{o}czky et al. \cite{B-L-Y-Z2012} extended
the $L_p$-Minkowski combination to $p>0$, that is,
\begin{equation*}%\label{loge1.2}
(1-\lambda)\cdot K +_{p}\, \lambda\cdot L
=\bigcap_{u\in S^{n-1}}\left\{x\in \mathbb{R}^n\mid x\cdot u\leq\big((1-\lambda)h_K(u)^p+\lambda h_L(u)^p\big)^{\frac{1}{p}}\right\},
\end{equation*}
where $h_K$ and $h_L$ are the support functions of $K$
and $L$. When $0<p<1$, the function $((1-\lambda)h_K^p+\lambda h_L^p)^{\frac{1}{p}}$
is not necessary the support function of the convex body
$(1-\lambda)\cdot K +_{p}\, \lambda\cdot L$, and
$(1-\lambda)\cdot K +_{p}\, \lambda\cdot L$ is the Wulff shape of
the function $((1-\lambda)h_K^p+\lambda h_L^p)^{\frac{1}{p}}$.
This is different from the case $p\ge 1$.
The limiting case of $p\rightarrow 0$
is the log Minkowski combination $(1-\lambda)\cdot K+_{o} \lambda \cdot L$,
which is defined by B\"{o}r\"{o}czky et al. in \cite{B-L-Y-Z2012}, that is,
\begin{equation*}%\label{loge1.3}
(1-\lambda)\cdot K+_{o} \lambda \cdot L=\bigcap_{u\in S^{n-1}}
\left\{x\in \mathbb{R}^n\mid x\cdot u\leq h_K(u)^{1-\lambda}h_L(u)^{\lambda}\right\}.
\end{equation*}

For two origin-symmetric convex bodies $K$ and $L$,
B\"{o}r\"{o}czky et al. \cite{B-L-Y-Z2012} conjectured that
the following log-Minkowski inequality holds
\begin{equation}\label{eqn1.1}
\int_{S^{n-1}}\log\frac{h_L}{h_K}\mathrm{d}\bar{V}_K
\geq \frac{1}{n}\log\frac{V(L)}{V(K)},
\end{equation}
where $\mathrm{d}\bar{V}_K$ is the cone-volume probability measure
of $K$ and showed that it is equivalent to the
log-Brunn-Minkowski inequality %(cf. \cite{B-L-Y-Z2012})
\begin{equation}\label{eqn1.2}
\mathrm{V}\big((1-\lambda)\cdot K+_{\circ} \lambda \cdot L\big)
\geq
\mathrm{V}(K)^{1-\lambda}\mathrm{V}(L)^{\lambda},
\quad \text{for $\lambda\in [0, 1]$.}
\end{equation}

B\"{o}r\"{o}czky et al. \cite{B-L-Y-Z2012} solved the planar case
for \eqref{eqn1.1} and \eqref{eqn1.2} and obtained the equalities
hold if and only if $K$ and $L$ are dilates or $K$ and $L$
are parallelograms with parallel sides. Ma \cite{M2015} gave an
alternative proof of \eqref{eqn1.1} for the case $n=2$. Saroglou
\cite{S2015} established \eqref{eqn1.2} together with its
equality cases for pairs of convex bodies that are both
unconditional with respect to some orthonormal basis.
Stancu \cite{S2016} showed some variants of the logarithmic
Minkowski inequality for general convex bodies and
obtained some special cases for the equality holds in
\eqref{eqn1.1} without the symmetric assumption.
Xi and Leng \cite{X-L2016} solved Dar's conjecture in the plane
and built the relationship between the log-Brunn-Minkowksi
inequality and Dar's conjecture in the plane when convex
bodies are at a dilation position.
To conclude \eqref{eqn1.1}, B\"{o}r\"{o}czky et al. \cite{B-L-Y-Z2012}
researched the uniqueness question of the cone-volume
measure for origin-symmetric convex bodies in the plane.
For the uniqueness of cone-volume measures,
Gage \cite{G1993} showed that within the class of
origin-symmetric planar convex bodies that are also smooth
and have positive curvature, the cone-volume measure
determines the convex body uniquely. For even discrete measures,
Stancu \cite{S2002,S2003} treated the
uniqueness question for the log-Minkowski problem
in the plane. There are many contexts in which
cone-volume measures play a significant role, see e.g.,
\cite{B-H2016,B-L-Y-Z2013,G-L1994,H-L2014,H-L-L2006,H-S-W2005,X2010,Z2014} etc..
Recently, a more comprehensive account of various aspects of the log-Brunn-Minkowski
inequality can be found in Colesanti-Livshyts-Marsiglietti \cite{C-L-M2017},
Colesanti-Livshyts \cite{C-L2017}, Kolesnikov-Milman \cite{K-M2017} and Rotem \cite{R2014}.

%%Another problem with the cone-volume measure is the existence
%%of logarithmic Minkowski problem or $L_0$-Minkowski problem,
%%that is to say, one needs to find a necessary and sufficient condition on a finite
%%Borel measure $\mu$ on the unit sphere $S^{n-1}$ so that
%%$\mu$ is the cone-volume measure of a convex body in $\mathbb{R}^n$.
%%A solution for the planar case can be found in Gage \cite{G1993},
%%Gage and Li \cite{G-L1994}, and Stancu \cite{S2002,S2003}.
%%B\"{o}r\"{o}czky et al. \cite{B-L-Y-Z2013}
%%dealt with the case of even measure and they showed that an even
%%measure on $S^{n-1}$ is the cone-volume measure of a symmetric
%%convex body in $\mathbb{R}^n$ if and only if it satisfies the
%%subspace concentration condition (cf. \cite{H-L-L2006,H-S-W2005,X2010}).
%%B\"{o}r\"{o}czky et al. \cite{B-H-Z2016} and Zhu \cite{Z2014}
%%considered the discrete case in $\mathbb{R}^n$.
%%There is also a variety of contexts in which
%%cone-volume measures play a significant role, see e.g.,
%%\cite{B-H2016,H-L2014} etc..

This paper is organized as follows. In Sect. \ref{sec2}, we present
some concepts and basic results about convex bodies. In Sect. \ref{sec3},
motivated by the idea of B\"{o}r\"{o}czky-Lutwak-Yang-Zhang
\cite{B-L-Y-Z2012} and Ma \cite{M2015}, we prove the log-Minkowski
inequality and log-Brunn-Minkowski inequality
when convex bodies are in $\mathfrak{R}_1$ class. For
the same convex bodies, the $L_p$-Minkowski inequality and
$L_p$-Brunn-Minkowski inequality are obtained  when $0<p<1$.
In Sect. \ref{sec4}, we show that there are convex bodies
such that they satisfy the condition of $\mathfrak{R}_i$ class.

\section{Preliminaries}\label{sec2}

In this section, we review some basic notations and definitions in
convex geometry. Good general references for the theory of convex
bodies are provided by the books of Gardner \cite{G2002}, Gruber
\cite{G2007}, Schneider \cite{S2014} and Thompson \cite{T1996}.

If $K\in \mathcal{K}^n$, its {\it support function}
$h_K:\mathbb{R}^n\rightarrow \mathbb{R}$
is defined by
\begin{equation*}
    h_K(x)=\max\{x\cdot y\mid y\in K\}.
\end{equation*}

Let $K\in \mathcal{K}^n$. The {\it surface area measure} $S_K$
of $K$ is a Borel measure on $S^{n-1}$ defined for a Borel set
$\omega\subseteq S^{n-1}$ by
\begin{equation*}
    S_K(\omega)=\mathcal{H}^{n-1}(\nu_K^{-1}(\omega)),
\end{equation*}
where $\nu_K:\,\partial' K\rightarrow S^{n-1}$ is the Gauss map of $K$,
defined on $\partial' K$, the set of points of $\partial K$ that
have a unique outer unit normal, and $\mathcal{H}^{n-1}$ is the
$(n-1)$-dimensional Hausdorff measure.

For $K\in \mathcal{K}_0^n$, its {\it cone-volume measure}
$V_K$ is a Borel measure on the unit sphere $S^{n-1}$ defined for
a Borel set $\omega\subseteq S^{n-1}$ by
\begin{equation*}
    V_K(\omega)=\frac{1}{n}\int_{x\in \nu_K^{-1}(\omega)}
    x\cdot \nu_K(x) \mathrm{d}\mathcal{H}^{n-1}(x)
\end{equation*}
and thus
\begin{equation*}
    \mathrm{d}V_K=\frac{1}{n} h_K \mathrm{d}S_K.
\end{equation*}
Since,
\begin{equation}\label{eqn2.1}
    V(K)=\frac{1}{n}\int_{ S^{n-1}}h_K(u) \mathrm{d}S_K(u),
\end{equation}
the {\it cone-volume probability measure} $\bar{V}_K$ of $K$
is given by
\begin{equation*}
    \bar{V}_K=\frac{1}{V(K)}V_K.
\end{equation*}

%%For $K\in \mathcal{K}_0^n$, its {\it radial function} $\rho_K(\cdot)$
%%is defined by
%%\begin{equation*}
%%    \rho_K(x)=\max\{\lambda \ge 0\mid \lambda x\in K\}.
%%\end{equation*}

Let $K\in\mathcal{K}^n$, $L\in\mathcal{K}_{0}^n$. For $x\in K$, set
\begin{equation*}
    r_x(K,L)=\max\{t\ge 0\mid x+ tL\subseteq K\}
\end{equation*}
and
\begin{equation*}
    R_x(K,L)=\min\{t> 0\mid x+ tL\supseteq K\}.
\end{equation*}
Specially, for $K,\,L \in \mathcal{K}_0^n$ and the origin $o$, we have
\begin{equation}\label{eqn2.2}
    r_o(K,L)=\min_{u\in S^{n-1}} \frac{h_K(u)}{h_L(u)}%=\min_{u\in S^{n-1}} \frac{\rho_K(u)}{\rho_L(u)}
\end{equation}
and
\begin{equation}\label{eqn2.3}
    R_o(K,L)=\max_{u\in S^{n-1}} \frac{h_K(u)}{h_L(u)}.%=\max_{u\in S^{n-1}} \frac{\rho_K(u)}{\rho_L(u)}.
\end{equation}
Obviously, from the above expressions, it follows that
\begin{equation*}
    r_o(K,L)=\frac{1}{R_o(L,K)}.
\end{equation*}

For $K,\,L\in \mathcal{K}^n$, the relative Steiner  formula states
that the volume of the outer parallel body of $K$ with respect to
$L$, $K+tL$, is a polynomial of degree
$n$ in $t\ge 0$,
\begin{equation}\label{eqn2.4}
    V(K+t L)=\sum_{i=0}^n {n\choose i }W_i(K,L) t^i.
\end{equation}
The coefficients $W_i(K,L)$ are called {\it relative quermassintegrals}
of $K$ with respect to $L$, and they are a special case of the general
defined {\it mixed volumes} for which we refer to \cite[Ch.5.1]{S2014}.
In particular, we have $W_0(K,L)=V(K)$, $W_n(K,L)=V(L)$
and $W_i(K,L)=W_{n-i}(L,K)$. When $n=3$, $W_1(K,L)$ and
$W_2(K,L)$ can be expressed by (see \cite[(5.34)]{S2014})
\begin{align}
    &W_1(K,L)=\frac{1}{3}\int_{S^2}h_L(u) \mathrm{d}S_K(u),\label{eqn2.5}\\
    &W_2(K,L)=\frac{1}{3}\int_{S^2}h_K(u) \mathrm{d}S_L(u).\label{eqn2.6}
\end{align}

Analogous formulae to \eqref{eqn2.4} give us the value of the relative
$i$-th quermassintegral of $K+tL$, namely
\begin{equation}\label{eqn2.7}
    W_i(K+t L,L)=\sum_{k=0}^{n-i} {n-i\choose k}W_{i+k}(K,L) t^k,
\end{equation}
for $t\ge 0$ and $i=0,\cdots,n$.% (cf. \cite{S2014}).
%%$r(K,L)=\max\{t>0\mid x+ tL\subseteq K \ \text{and}\ x\in \mathbb{R}^n\}$,
%%
%%$R(K,L)=\min\{t>0\mid x+ tL\supseteq K \ \text{and}\ x\in \mathbb{R}^n\}$.

To introduce the convex bodies that are in $\mathcal{R}_i$ class,
for $K,\,L\in \mathcal{K}^n$, we consider the {\it $i$-th
relative Bonnesen function}
\begin{equation}\label{eqn2.8}
    \mathcal{B}_{i;K,L}(r)=2W_{i+1}(K,L)r-W_i(K,L)-W_{i+2}(K,L)r^2.
\end{equation}

Next, we give a proposition about the $i$-th
relative Bonnesen function $\mathcal{B}_{i;K,L}(r)$.

\begin{proposition}\label{pro2.1}
Let $K\in \mathcal{K}^3$, $L\in \mathcal{K}_0^3$. If $K_t=K+tL\,(t\ge 0)$ are outer parallel bodies
of $K$ with respect to $L$,
%If $\widetilde{\mathcal{B}}_i(r)$ is the relative $i$-th Bonnesen function of $K_t$,
then, for $i=0,1,\cdots,n-2$,
\begin{equation}\label{eqn3.2}
    \mathcal{B}_{i;K_t,L}(r+t)=\sum_{k=0}^{n-i-2} {n-i-2 \choose k} \mathcal{B}_{i+k;K,L}(r)t^k.
\end{equation}
\end{proposition}

\begin{proof}
It follows from \eqref{eqn2.7} and \eqref{eqn2.8} that
\begin{align*}
    \mathcal{B}_{i;K_t,L}(r)=&2W_{i+1}(K+t L,L)r-W_i(K+tL,L)-W_{i+2}(K+t L,L)r^2\\
    =&2\sum_{k=0}^{n-i-1} {n-i-1 \choose k} W_{i+k+1}(K,L) t^k r-
    \sum_{k=0}^{n-i} {n-i \choose k} W_{i+k}(K,L) t^k\\
    &-\sum_{k=0}^{n-i-2} {n-i-2 \choose k} W_{i+k+2}(K,L)t^k r^2.
\end{align*}
Set
\begin{align*}
   A=&\sum_{k=0}^{n-i-1}{n-i-1 \choose k} W_{i+k+1}(K,L)t^k
    -\sum_{k=0}^{n-i-2}{n-i-2 \choose k} W_{i+k+2}(K,L)t^{k+1},\\
   B=&2\sum_{k=0}^{n-i-1}{n-i-1 \choose k}W_{i+k+1}(K,L)t^{k+1}
    -\sum_{k=0}^{n-i}{n-i \choose k}W_{i+k}(K,L)t^{k}\\
    &-\sum_{k=0}^{n-i-2}{n-i-2 \choose k}W_{i+k+2}(K,L)t^{k+2},
\end{align*}
then
\begin{align*}
    \mathcal{B}_{i;K_t,L}(r+t)%=&2W_{i+1}(K+t B_n)r-W_i(K+tB_n)-W_{i+2}(K+t B_n)r^2\\
    =&2\sum_{k=0}^{n-i-1} {n-i-1 \choose k} W_{i+k+1}(K,L) t^k (r+t)-
    \sum_{k=0}^{n-i} {n-i \choose k} W_{i+k}(K,L) t^k\\
    &-\sum_{k=0}^{n-i-2} {n-i-2 \choose k} W_{i+k+2}(K,L)t^k(r+t)^2\\
    =&2Ar+B-\sum_{k=0}^{n-i-2}{n-i-2 \choose k} W_{i+k+2}(K,L)t^k r^2.
%    =&2\left(\sum_{k=0}^{n-i-1}{n-i-1 \choose k} W_{i+k+1}(K,L)t^k
%    -\sum_{k=0}^{n-i-2}{n-i-2 \choose k} W_{i+k+2}(K,L)t^{k+1}\right)r\\
%    &+2\sum_{k=0}^{n-i-1}{n-i-1 \choose k}W_{i+k+1}(K,L)t^{k+1}
%    -\sum_{k=0}^{n-i}{n-i \choose k}W_{i+k}(K,L)t^{k}\\
%    &-\sum_{k=0}^{n-i-2}{n-i-2 \choose k}W_{i+k+2}(K,L)t^{k+2}
%    -\sum_{k=0}^{n-i-2}{n-i-2 \choose k} W_{i+k+2}(K,L)t^k r^2\\
%    \triangleq& 2Ar+B-\sum_{k=0}^{n-i-2}{n-i-2 \choose k} W_{i+k+2}(K,L)t^k r^2.
\end{align*}
The expressions $A$ and $B$ can be simplified as
\begin{align*}
    A&=\sum_{k=0}^{n-i-2}{n-i-2 \choose k} W_{i+k+1}(K,L)t^k,\\
    B&=-\sum_{k=0}^{n-i-2}{n-i-2 \choose k} W_{i+k}(K,L)t^{k},
\end{align*}
hence,
\begin{align*}
    \mathcal{B}_{i;K_t,L}(r+t)&=\sum_{k=0}^{n-i-2}{n-i-2 \choose k}\left(2W_{i+k+1}(K,L)r-W_{i+k}(K,L)-W_{i+k+2}(K,L)r^2\right)t^k\\
    &=\sum_{k=0}^{n-i-2}{n-i-2 \choose k} \mathcal{B}_{i+k;K,L}(r)t^k,
\end{align*}
which completes the proof.
\end{proof}

For convex bodies $K,\,L\in \mathcal{K}^3$, in order to
research the log-Brunn-Minkowski inequality and the
log-Minkowski inequality in $\mathbb{R}^3$, we give
the definitions that convex bodies are in
$\mathfrak{R}_1$ class and
$\mathfrak{R}_2$ class.

\begin{definition}\label{def2.2}
Let $K,\,L\in \mathcal{K}^3$. The convex body $K$ is in
$\mathfrak{R}_1$ class with respect to $L$
if the origin $o\in \mathrm{int} (K\cap L)$ such that
$$\mathcal{B}_{0;K,L}(r)\ge 0,\quad r\in[r_o(K,L),R_o(K,L)];$$
The convex body $K$ is in $\mathfrak{R}_2$ class with respect to
$L$ if the origin $o\in \mathrm{int} (K\cap L)$ such that
$$\mathcal{B}_{1;K,L}(r)\ge 0,\quad r\in[r_o(K,L),R_o(K,L)].$$
\end{definition}

From Definition \ref{def2.2} and Proposition \ref{pro2.1},
we have the following corollary.

\begin{corollary}\label{cor2.3}
Let $K,\,L\in \mathcal{K}^3$.
\begin{itemize}
  \item [(i)] If $K$ is in $\mathfrak{R}_1$ class
    with respect to $L$, then $L$ is in $\mathfrak{R}_2$ class
    with respect to $K$.
  \item[(ii)] If $K$ is in $\mathfrak{R}_2$ class
    with respect to $L$, then $L$ is in $\mathfrak{R}_1$ class
    with respect to $K$.
  \item [(iii)] If $K$ is in $\mathfrak{R}_1$ class as well as in $\mathfrak{R}_2$ class
    with respect to $L$, then $K+tL$ is in $\mathfrak{R}_1$ class
    with respect to $L$.
  \item [(iv)] If $K$ is in $\mathfrak{R}_2$ class
    with respect to $L$, then $K+tL$ is also in $\mathfrak{R}_2$ class
    with respect to $L$.
\end{itemize}
\end{corollary}

\section{The log-Brunn-Minkowski inequality and the log-Minkowski inequality}\label{sec3}

In this section, inspired by the impressive work of B\"{o}r\"{o}czky-Lutwak-Yang-Zhang
\cite{B-L-Y-Z2012} and Ma \cite{M2015}, firstly,
we deal with the log-Brunn-Minkowski inequality and the log-Minkowski inequality
for two special convex bodies. Secondly, we obtain the $L_p$-Minkowski inequality
and the $L_p$-Brunn-Minkowski inequality for $0<p<1$.

\begin{lemma}\label{lem3.1}
Let $K,\,L\in\mathcal{K}^3$. If $K$ is in $\mathfrak{R}_1$ class
with respect to $L$, then
\begin{equation}\label{eqn3.1}
   \int_{S^{2}} \frac{h^2_K}{h_L}\mathrm{d}S_{K}\le\frac{3V(K)W_1(K,L)}{W_2(K,L)};
\end{equation}
If $K$ is $\mathfrak{R}_2$ class
with respect to $L$, then
\begin{equation}\label{eqn3.2}
   \int_{S^{2}} \frac{h^2_K}{h_L}\mathrm{d}S_{K}\le\frac{6V(K)W_2(K,L)-3W_1(K,L)^2}{V(L)},
\end{equation}
and equalities in \eqref{eqn3.1} and \eqref{eqn3.2} hold when $K$ and $L$ are dilates.
\end{lemma}

\begin{proof}
If $K$ is in $\mathfrak{R}_1$ class with respect to $L$, then
\begin{equation*}
    r_o(K,L)\le\frac{h_K(u)}{h_L(u)}\le R_o(K,L),
\end{equation*}
for all $u\in S^2$. Thus, from \eqref{eqn2.8} and Definition \ref{def2.2},
we have
\begin{equation*}
    2W_1(K,L)\frac{h_K(u)}{h_L(u)}-V(K)-W_2(K,L)\left(\frac{h_K(u)}{h_L(u)}\right)^2\ge 0.
\end{equation*}
Integrating this with respect to the measure $h_L \mathrm{d}S_{K}$, and using
\eqref{eqn2.1}, \eqref{eqn2.5} and \eqref{eqn2.6}, give us
\begin{align*}
    0\le& \int_{S^2} \left(2W_1(K,L)\frac{h_K(u)}{h_L(u)}-V(K)-W_2(K,L)\left(\frac{h_K(u)}{h_L(u)}\right)^2\right)
    h_L(u) \mathrm{d}S_{K}(u) \\
      =&3 V(K)W_1(K,L)-W_2(K,L)\int_{S^2} \frac{h_K(u)^2}{h_L(u)} \mathrm{d}S_{K}(u),
\end{align*}
which yields the desired inequality \eqref{eqn3.1}.

It is obvious that the equality in \eqref{eqn3.1} holds when $K$ and $L$ are dilates.

Similarly, inequality \eqref{eqn3.2} is obtained when
$K$ is in $\mathfrak{R}_2$ class with respect to $L$.
%With a similar method, we can get \eqref{eqn3.2} when
%$K$ is $\mathfrak{R}_2$ class with respect to $L$.
\end{proof}

\begin{theorem}\label{thm3.2}
Let $K,\,L\in\mathcal{K}^3$. If $K$ is in $\mathfrak{R}_1$ class
with respect to $L$, then
\begin{equation}\label{eqn3.3}
   \int_{S^2}\log\left(\frac{h_L}{h_K}\right) \mathrm{d}V_{K}\ge
   V(K)\log\left(\frac{V(L)}{V(K)}\right)^{\frac{1}{3}},
\end{equation}
with equality holds when $K$ and $L$ are dilates.
\end{theorem}

\begin{proof}
Since $K$ is in $\mathfrak{R}_1$ class with respect to $L$,
from Corollary \ref{cor2.3} (i), (ii) and (iv), we know that
$K$ is in $\mathfrak{R}_1$ class with respect to $L+tK$.
Thus, from \eqref{eqn3.1} and the Aleksandrov-Fenchel
inequality, we get
\begin{align}
    \int_{S^{2}} \frac{h_K}{h_{L+tK}}\mathrm{d}V_{K}
    &\le\frac{V(K)W_1(K, L+tK)}{W_2(K, L+tK)}\nonumber\\
    &\le\frac{V(K)W_2(K, L+tK)}{V(L+tK)}
    =\frac{V(K)W_1(L+tK,K)}{V(L+tK)}.\label{eqn3.4}
\end{align}
Let
\begin{equation*}
    F(t)=\int_{S^{2}} \log\left(\frac{h_{L+tK}}{h_K}\right)\mathrm{d}V_{K}
    -V(K)\log\left(\frac{V(L+tK)}{V(K)}\right)^{\frac{1}{3}}.
\end{equation*}
Differentiating $F(t)$ with respect to $t$, we have
\begin{align*}
    F'(t)=&\frac{\mathrm{d}}{\mathrm{d}t}\left(\int_{S^2}\log\left(\frac{h_{L+tK}}{h_K}\right)\mathrm{d}V_K
    -V(K)\log\left(\frac{V(L+tK)}{V(K)}\right)^{\frac{1}{3}}\right)\\
     =&\frac{\mathrm{d}}{\mathrm{d}t}\Bigg(\int_{S^2}\log\left(\frac{h_{L}+t h_{K}}{h_K}\right)\mathrm{d}V_K\\
     &-V(K)\log\left(\frac{V(L)+3W_1(L,K)t+3W_2(L,K)t^2+V(K)t^3}{V(K)}\right)^{\frac{1}{3}}\Bigg)\\
     =&\int_{S^2} \frac{h_K}{h_L+t h_K} \mathrm{d}V_K-\frac{V(K)(W_1(L,K)+2W_2(L,K)t+V(K)t^2)}{V(L)+3W_1(L,K)t+3W_2(L,K)t^2+V(K)t^3}\\
     =&\int_{S^2} \frac{h_K}{h_{L+tK}} \mathrm{d}V_K-\frac{V(K)W_1(L+tK,K)}{V(L+tK)}.
\end{align*}
It follows from \eqref{eqn3.4} that $F(t)$ is decreasing on $[0,+\infty)$.
The desired result can be achieved if we show that $F(0)\ge F(t)\ge 0$
for $t\in [0,+\infty)$. Since
\begin{align*}
    F(t)&=\int_{S^{2}} \log\left(\frac{h_{L+tK}}{h_K}\right)\mathrm{d}V_{K}
    -\int_{S^2}\log\left(\frac{V(L+tK)}{V(K)}\right)^{\frac{1}{3}} \mathrm{d}V_{K}\\
    &=\int_{S^{2}}\log\left(\frac{h_{L+tK}}{h_K}\left(\frac{V(K)}{V(L+tK)}\right)^{\frac{1}{3}}\right)  \mathrm{d}V_{K},
\end{align*}
by the mean value theorem for integrals there exists $\mu\in S^2$ such that
\begin{equation*}
    F(t)=V(K)\log\left(\frac{h_L(\mu)+t h_K(\mu)}{h_K(\mu)}
    \left(\frac{V(K)}{V(L)+3W_1(L,K)t+3W_2(L,K)t^2+V(K)t^3}\right)^{\frac{1}{3}}\right).
\end{equation*}
Then
\begin{equation*}
    \lim_{t\rightarrow +\infty}F(t)=0,
%    V(K)\lim_{t\rightarrow +\infty}
%    \log\left(\frac{h_L(\mu)+t h_K(\mu)}{h_K(\mu)}
%    \left(\frac{V(K)}{V(K)+3W_1(L,K)t+3W_2(L,K)t^2+V(K)t^3}\right)^{\frac{1}{3}}\right)=0,
\end{equation*}
which implies $F(0)\ge F(t)\ge 0$ for $t\in [0,+\infty)$.

It is obvious that the equality in \eqref{eqn3.3} holds
when $K$ and $L$ are dilates.
\end{proof}

\begin{theorem}\label{thm3.3}
Let $K,\,L\in\mathcal{K}^3$. If $K$ is in $\mathfrak{R}_1$ class
with respect to $L$, then, for $\lambda\in[0,1]$,
\begin{equation}\label{eqn3.5}
V((1-\lambda)\cdot K+_o \lambda \cdot L)\ge V(K)^{1-\lambda}V(L)^{\lambda}.
\end{equation}
When $\lambda\in (0,1)$, the equality in \eqref{eqn3.5} holds if
$K$ and $L$ are dilates.
\end{theorem}

%%To complete the proof of this theorem, we need the following lemma
%%which is achieved by the same method as \cite[Lemma 3.2]{B-L-Y-Z2012}.
%%Here, we omit its detailed proof.

%%\begin{lemma}\label{lem3.4}
%%Let $K,\,L\in\mathcal{K}^3$. If $K$ is $\mathfrak{R}_1$ class
%%with respect to $L$, then the log-Brunn-Minkowski inequality
%%\eqref{eqn3.5} and the log-Minkowski inequality \eqref{eqn3.3}
%%are equivalent.
%%\end{lemma}

%%{\noindent\bf{Proof of Theorem \ref{thm3.3}}}
%%\quad Lemma \ref{lem3.4}, together with \eqref{eqn3.3}, gives us the desired result.
%%And it is clear that the equality in \eqref{eqn3.5} holds for $\lambda\in(0,1)$
%%when $K$ and $L$ are dilates.
%%\qed

\begin{proof}
By the same method as \cite[Lemma 3.2]{B-L-Y-Z2012}, we have
the log-Brunn-Minkowski inequality \eqref{eqn3.5} and the
log-Minkowski inequality \eqref{eqn3.3} are equivalent.
And it is clear that the equality in \eqref{eqn3.5} holds for $\lambda\in(0,1)$
when $K$ and $L$ are dilates.
\end{proof}

\begin{remark}\label{rem3.4}
Theorems \ref{thm3.2} and \ref{thm3.3} implies that
inequalities \eqref{eqn3.3} and \eqref{eqn3.5} hold when
$K$ is in $\mathfrak{R}_1$ class with respect to $L$,
and Saroglou \cite{S2015} proved that inequalities
\eqref{eqn3.3} and \eqref{eqn3.5} hold when
$K$ and $L$ are unconditional. Next, we show that
these two conditions don't include each other.

Let $K$ be the cube of edge $2$ and $L$ the rectangular
parallelepiped whose concurrent edges have length $2$,
$2$ and $4$. Suppose that $K$ and $L$ are
symmetric with respect to the origin. It is
obvious that $K$ and $L$ are unconditional
with respect to orthonormal basis.
By a simple computation, we have
$V(K)=8$, $V(L)=16$, $W_1(K,L)=\frac{32}{3}$,
$W_2(K,L)=\frac{40}{3}$, $r(K,L)=\frac{1}{2}$ and
$R(K,L)=1$, which implies that $K$ is not in $\mathfrak{R}_1$
class with respect to $L$.

From the proof of Proposition \ref{pro4.1},
we construct convex bodies such that they are
in $\mathfrak{R}_1$ class rather than 
unconditional with respect to orthonormal basis.

%%Let $K$ be the cube of edge $2$ and $L$ the rectangular
%%parallelepiped whose concurrent edges have length $2$,
%%$2$ and $4$. Suppose that $K$ and $L$ are
%%symmetric with respect to the origin.
%%From the argument in Saroglou \cite{S2015}, we know
%%that equalities in \eqref{eqn3.3} and \eqref{eqn3.5}
%%hold. By some simple computation, we have known that
%%$V(K)=8$, $V(L)=16$, $W_1(K,L)=\frac{32}{3}$,
%%$W_2(K,L)=\frac{40}{3}$, $r(K,L)=\frac{1}{2}$ and
%%$R(K,L)=1$, which implies that $K$ is not $\mathfrak{R}_1$
%%class with respect to $L$ and the equality in \eqref{eqn3.1} doesn't
%%hold. Hence, compared with the planar case
%%(see \cite{B-L-Y-Z2012,X-L2016}),
%%we can not obtain all possibilities satisfying
%%equalities in \eqref{eqn3.3} and \eqref{eqn3.5}
%%by Lemma \ref{lem3.1} under our conditions.
\end{remark}

As natural extensions of the log-Minkowski inequality
and the log-Brunn-Minkowski inequality, we have the
$L_p$-Minkowski inequality and the $L_p$-Brunn-Minkowski inequality
using the same method which is appeared in \cite{B-L-Y-Z2012}.

\begin{theorem}\label{thm3.5}
Let $K,\,L\in\mathcal{K}^3$ and $p\in (0,1)$.
If $K$ is in $\mathfrak{R}_1$ class with respect to $L$,
then
\begin{equation}\label{eqn3.6}
    \left(\int_{S^2}\left(\frac{h_L}{h_K}\right)^p \mathrm{d}\bar{V}_K\right)^{\frac{1}{p}}
    \ge \left(\frac{V(L)}{V(K)}\right)^{\frac{1}{3}},
\end{equation}
with equality holds if $K$ and $L$ are dilates.
\end{theorem}

\begin{proof}
Jensen's inequality, together with the log-Minkowski inequality \eqref{eqn3.3},
shows that the $L_p$-Minkowski inequality \eqref{eqn3.6}, for $p>0$. When
$K$ and $L$ are dilates, the equality in \eqref{eqn3.6} holds.
\end{proof}

\begin{theorem}\label{thm3.6}
Let $K,\,L\in\mathcal{K}^3$ and $p\in (0,1)$.
If $K$ is in $\mathfrak{R}_1$ class with respect to $L$,
then
\begin{equation}\label{eqn3.7}
    V((1-\lambda)\cdot K+_p \lambda \cdot L)\ge V(K)^{1-\lambda}V(L)^{\lambda},
\end{equation}
with equality holds if $K=L$.
\end{theorem}

%%Similar with \cite[Lemma 3.1]{B-L-Y-Z2012}, we have the following lemma.
%%Here, we omit its detailed proof.
%%\begin{lemma}\label{lem3.8}
%%Let $K,\,L\in\mathcal{K}^3$ and $p\in (0,1)$. If $K$ is $\mathfrak{R}_1$ class
%%with respect to $L$, then the $L_p$-Brunn-Minkowski inequality
%%\eqref{eqn3.7} and the $L_p$-Minkowski inequality \eqref{eqn3.6}
%%are equivalent.
%%\end{lemma}

%%{\noindent\bf{Proof of Theorem \ref{thm3.7}}}
%%\quad From Lemma \ref{lem3.8} and the $L_p$-Minkowski inequality \eqref{eqn3.6},
%%we have the $L_p$-Brunn-Minkowski inequality \eqref{eqn3.7}. It is obvious
%%that the equality in \eqref{eqn3.7} holds when $K=L$.
%%\qed

\begin{proof}
Similar with \cite[Lemma 3.1]{B-L-Y-Z2012}, we have the $L_p$-Brunn-Minkowski inequality
\eqref{eqn3.7} and the $L_p$-Minkowski inequality \eqref{eqn3.6} are equivalent.
Meanwhile, it is obvious that the equality in \eqref{eqn3.7} holds when $K=L$.
\end{proof}

\section{The existence of the convex bodies that are in $\mathfrak{R}_i$ class}\label{sec4}

In this section, we show the existence of three-dimensional convex bodies that are in
$\mathfrak{R}_i$ class. %in $\mathbb{R}^3$.

\begin{proposition}\label{pro4.1}
There exist convex bodies $K_1,\,K_2\in\mathcal{K}_0^3$ such
that $K_1$ is not only in $\mathfrak{R}_1$ class but also in
$\mathfrak{R}_2$ class with respect to $K_2$.
\end{proposition}

\begin{proof}
We construct some revolutionary constant width convex bodies to
satisfy these conditions. Let $D_i$ be a domain with support function
%$h(\theta)=\frac{1}{2}+\frac{1}{12(n^2+n)}\sin(2n+1)\theta$.
\begin{equation}\label{eqn4.1}
    h_i(\theta)=\frac{1}{2}+\frac{1}{12(n_i^2+n_i)}\sin(2n_i+1)\theta,
    \quad i=1,\,2.
\end{equation}
Since $h_i(\theta)+h_i(\theta+\pi)=1$
and $h_i(\theta)+h_i''(\theta)=\frac{1}{2}-\frac{1}{3}\sin(2n_i+1)\theta>0$,
$D_i$ is a planar convex body of constant width $1$ and symmetric with
respect to the $y$-axis (see Figure \ref{pic1-1}).
Denote by $K_i$ the convex body by rotating $D_i$ around the $y$-axis (see Figure \ref{pic1-2}).
It follows from \cite[Proposition 3.4]{P-Z-Y2016} that $K_i$ is a revolutionary body
of constant width $1$, which together with \eqref{eqn2.5} and \eqref{eqn2.6} implies
\begin{equation*}
    W_1(K_1,K_2)=\frac{S(K_1)}{6}\quad \text{and}\quad W_2(K_1,K_2)=\frac{S(K_2)}{6},
\end{equation*}
where $S(K_i)$ is the surface area of $K_i$.
The surface area of $K_i$ can be expressed by (see \cite[p.1086 (28)]{B-H2012})
\begin{equation}\label{eqn4.2}
    S(K_i)=2\pi\int_{-\frac{\pi}{2}}^{\frac{\pi}{2}}
    \left(h_i(\theta)^2-\frac{1}{2}h_i'(\theta)^2\right)\cos\theta\mathrm{d}\theta.
\end{equation}
By the Blaschke identity (cf. \cite{B1915} or \cite[Theorem 4]{C1966})
\begin{equation}\label{eqn4.3}
    2V(K)=w S(K)-\frac{2\pi}{3}w^3,
\end{equation}
where $w$ is the width of $K$, we can get the volume of $K_1$ and $K_2$.

In order to prove that $K_1$ is in $\mathfrak{R}_1$ class
as well as in $\mathfrak{R}_2$ class with respect to $K_2$,
we have to show that the origin $o$ satisfies the conditions
of Definition \ref{def2.2}. Since $r_o(D_1,D_2)=r_o(K_1,K_2)$
and $R_o(D_1,D_2)=R_o(K_1,K_2)$, we need only to
ensure the value of $r_o(D_1,D_2)$ and $R_o(D_1,D_2)$.
%%Let $\mathcal{D}_i(\theta)$ be the square of the radial function
%%of $D_i$.
%%Then, we have
%%\begin{align*}
%%    \mathcal{D}_i(\theta)&=\rho_{D_i}(\theta)^2=x_i(\theta)^2+y_i(\theta)^2=h_i(\theta)^2+h_i'(\theta)^2\\
%%    &=\frac{1}{4}+\frac{1}{144(n_i^2+n_i)^2}
%%    +\frac{1}{12(n_i^2+n_i)}\sin(2n_i+1)\theta+\frac{1}{36(n_i^2+n_i)}(\cos(2n_i+1)\theta)^2,
%%\end{align*}
%%where $x_i(\theta)=h_i(\theta)\cos\theta-h_i'(\theta)\sin\theta$ and
%%$y_i(\theta)=h_i(\theta)\sin\theta+h_i'(\theta)\cos\theta$, $i=1,\,2$.
%%To ensure $r_o(D_1,D_2)$ and $R_o(D_1,D_2)$, from \eqref{eqn2.2} and \eqref{eqn2.3},
%%we need only to find the maximum and minimum of $\frac{\mathcal{D}_1(\theta)}{\mathcal{D}_2(\theta)}$.

According to \eqref{eqn2.2}, \eqref{eqn2.3}, \eqref{eqn4.1},
\eqref{eqn4.2} and \eqref{eqn4.3}, we can do some numerical computation.
These results show that $\mathcal{B}_i(r_o(D_1,D_2))>0$
and $\mathcal{B}_i(R_o(D_1,D_2))>0$ for $i=0,\,1$,
when $n_i$ chooses different values, such as $n_1=1,\,n_2=2$; $n_1=3,\,n_2=5$; etc..
\end{proof}

\begin{figure}%[H]
\begin{center}
\subfloat[$D_i$ \label{pic1-1}]{{\includegraphics[width=0.38\textwidth]{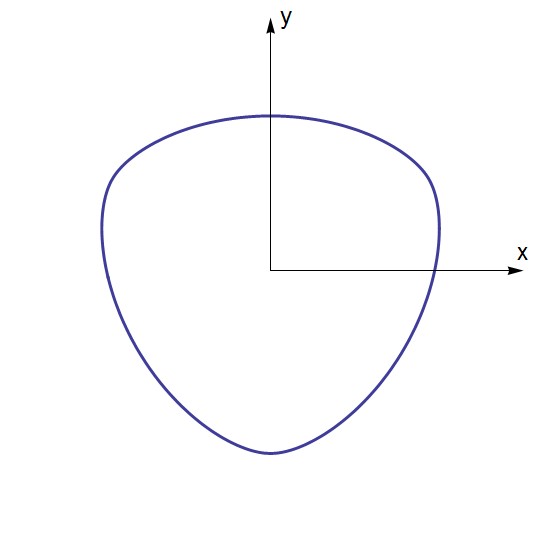}}}\qquad
\subfloat[$K_i$ \label{pic1-2}]{{\includegraphics[width=0.5\textwidth]{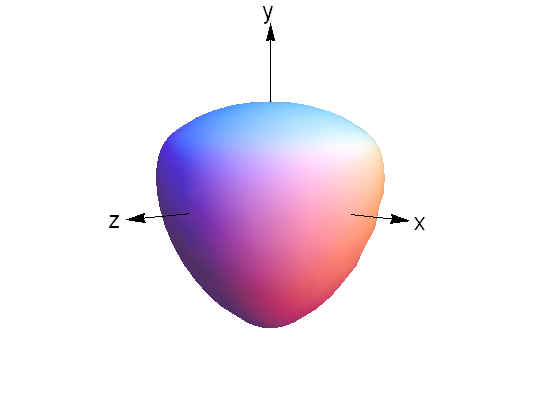}}}
\caption{Constant width domain $D_i$ and its revolutionary body $K_i$
\label{pic1}}
\end{center}
\end{figure}

Next, we will give an example of pairs of three-dimensional origin symmetric convex bodies
that are in $\mathfrak{R}_1$ class.

\begin{example}\label{ex4.2}
Let $K$ be a cylinder of height $2$ and radius $1$ and $L$ a unit ball. Suppose that $K$ and $L$ are symmetric with
respect to the origin. From \cite[1(d) p230]{S2004}, it follows that
$W_0(K,L)=2\pi$, $W_1(K,L)=2\pi$ and $W_2(K,L)=\frac{\pi(\pi+2)}{3}$.
It is clear that $r_o(K,L)=r(K,L)=1$ and $R_o(K,L)=R(K,L)=\sqrt{2}$.
Some simple computations show that
$K$ is in $\mathfrak{R}_1$ class with respect to $L$.
\end{example}

%\iffalse
\section*{Acknowledgements}
%%This work is supported by the National Science Foundation
%%of China (No. 11171254).
We are grateful to the anonymous referee for his or her careful
reading of the original manuscript of this paper and giving us some
invaluable comments.
%\fi

\bibliographystyle{amsplain}

\end{document}